\documentclass{amsart}

\usepackage[english]{babel}

\usepackage[letterpaper,top=2cm,bottom=2cm,left=3cm,right=3cm,marginparwidth=1.75cm]{geometry}

\usepackage{amsmath}
\usepackage{amssymb}
\usepackage{graphicx}
\usepackage[pdftex]{pict2e}
\usepackage{enumitem}
\usepackage{tikz}

\newtheorem{theorem}{Theorem}[section] 
\newtheorem{lemma}[theorem]{Lemma}     
\newtheorem{corollary}[theorem]{Corollary}
\newtheorem{proposition}[theorem]{Proposition}

\theoremstyle{definition}
\newtheorem{example}[theorem]{Example}
\newtheorem{remark}[theorem]{Remark}
\newtheorem{Def}[theorem]{Definition}

\newcommand\lu{\mathbb{L}}

\newcommand\alge{\mathfrak{A}}
\newcommand\algeq{\mathfrak{Q}}
\newcommand\tr{\textnormal{ tr}}

\newcommand\oink{\mathcal O}

\newcommand\Ba{\mathbf B}
\newcommand\Ra{\mathbf R}
\newcommand\Da{\mathbf{D}}
\newcommand\Ha{\mathbf{H}}
\newcommand\Ea{\mathbf{E}}

\newcommand\cs{{}_{\clubsuit}}
\newcommand\sps{{}_{\spadesuit}}

\newcommand\bmattrix[4]{\left(\begin{array}{cc}#1&#2\\#3&#4\end{array}\right)}
\newcommand\sbmattrix[4]{\textnormal{\scriptsize$\left(\begin{array}{cc}#1&#2\\#3&#4\end{array}\right)$\normalsize}}

\title[Bass orders in division algebras]
{On Bass orders in division algebras}

\author{Luis Arenas-Carmona}

\begin{document}

\footnote{\emph{Keywords:}
Bass orders, maximal orders, Bruhat-Tits trees.\\
\emph{MSC (2020):} 11S45(Primary), 16H10, 16G30 (Secondary).}

\begin{abstract}
A result by F.-T. Tu dating from 2011 states that the orders in a local
matrix algebra of dimension four that can be written as a finite 
intersection of maximal orders are those having an Eichler Gorenstein 
closure. In a more recent paper, we gave a similar characterization for 
orders with a Bass Gorenstein closure. More precisely, we prove that
these orders are the ones that can be written as a ghost intersection,
i.e., the set of points defined over the base field of an intersection
of maximal orders that is defined over a field extension. Although we proved 
this result only for the local four-dimensional matrix algebra, the 
statement makes sense also for a local quaternion division algebra,
since the latter turns into a matrix algebra when we extend the base field
to an even degree extension. In this work, we prove an analog of our 
previous result in the latter setting.
\end{abstract}

\maketitle

\section{Introduction}\label{intro}

In all of this work, $K$ denotes a local field $K$ with ring of integers $\mathcal{O}_K$
and uniformizing parameter $\pi_K$.
We let $\alge$ be the unique quaternion division $K$-algebra, up to isomorphism, and write 
$\mathbf{1}=\mathbf{1}_{\alge}$ for its multiplicative identity. By a full
order $\Ra$ we mean an $\mathcal{O}_K$-order of maximal rank. The order
$\Ra$ is called Gorenstein when its codifferent, defined as
$$\mathrm{codiff}(\Ra)=\{a \in\alge| \tr(a\Ra)\subseteq\oink_K\}
\subseteq\alge,$$
is an invertible two-sided fractional $\Ra$-ideal. There are other
equivalent definitions of Gorenstein orders, see \cite[Prop. 24.2.3]{voight} 
for instance. The one that we use mostly is the following 
\cite[Prop. 24.2.15]{voight}:
\begin{quote}
    An order $\Ra$ is called Gorenstein if it cannot be written in the form
    $\Ra=\Ra_1^{[n]}:=\oink_K\mathbf{1}+\pi_K^{n}\Ra_1$, for any integer 
    $n\geq1$, and a full order $\Ra_1$.
\end{quote}
In fact, for any full order $\Ra'$ there is a largest integer $n$ such that we can 
write $\Ra'=\Ra^{[n]}$ as above. The order $\Ra$ is unique, and it is called the 
Gorenstein closure (sometimes Gorenstein saturation) of $\Ra'$. An analogous result,
replacing the power $\pi_K^n$ by an arbitrary ideal $I$, holds for Dedekind domains in global fields, but we do not need it here.
An order $\Ba$ is called Bass if every order $\Ra$ satisfying $\Ra\supseteq\Ba$ is 
Gorenstein.  Gorenstein and Bass orders are key to the classical classification of 
quaternion orders. See \cite{Brz82}, \cite{Brz83} and \cite[\S24.2]{voight}
for an account of this theory.

It has been known for some time that a full order in a local matrix algebra is the
intersection of a family of maximal orders if and only if its Gorenstein closure
is an Eichler order \cite{Brz83}. This is useful to study embeddings of arbitrary orders into 
Eichler orders, or other orders with an Eichler Gorenstein closure. This is done 
through the use of branches of orders. 
The Bruhat-Tits tree $\mathfrak{t}(K)$ is a tree whose
vertices correspond to the maximal orders in the matrix algebra 
$\mathbb{M}_2(K)$. The branch of an order 
$\Ha\subseteq\mathbb{M}_2(K)$ is the largest
subgraph $\mathfrak{s}(\Ha)$ of the Bruhat-Tits tree whose vertices correspond
to maximal orders containing $\Ha$. A typical branch is like the one depicted
in Fig. \ref{f1}. 
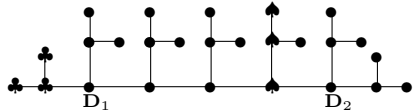
\begin{figure}
    \centering
    \setlength{\unitlength}{2mm}
    \begin{picture}(40,15)(-5,-3)
        \put(3,0){\line(1,0){26}}
        \put(5,0){\line(0,1){2}}
        \put(8,0){\line(0,1){5}}
        \put(8,3){\line(1,0){2}}
        \put(12,0){\line(0,1){5}}
        \put(12,3){\line(1,0){2}}
        \put(16,0){\line(0,1){5}}
        \put(16,3){\line(1,0){2}}
        \put(20,0){\line(0,1){5}}
        \put(20,3){\line(1,0){2}}
        \put(24,0){\line(0,1){5}}
        \put(24,3){\line(1,0){2}}
        \put(27,0){\line(0,1){2}}
        \put(11.5,-0.5){$\bullet$}
        \put(15.5,-0.5){$\bullet$}
        \put(19.5,0){$\sps$}
        \put(15.5,2.5){$\bullet$}
        \put(19.5,3){$\sps$}
        \put(23.5,2.5){$\bullet$}
        \put(19.5,5){$\sps$}
        \put(21.5,2.5){$\bullet$}
        \put(23.5,-0.5){$\bullet$}
        \put(26.5,-0.5){$\bullet$}
        \put(28.5,-0.5){$\bullet$}
        \put(26.5,1.5){$\bullet$}
        \put(25.5,2.5){$\bullet$}
        \put(23.5,4.5){$\bullet$}
        \put(17.5,2.5){$\bullet$}
        \put(15.5,4.5){$\bullet$}
        \put(13.5,2.5){$\bullet$}
        \put(11.5,4.5){$\bullet$}
        \put(9.5,2.5){$\bullet$}
        \put(7.5,4.5){$\bullet$}
        \put(7.5,2.5){$\bullet$}
        \put(11.5,2.5){$\bullet$}
        \put(7.5,2.5){$\bullet$}
        \put(11.5,2.5){$\bullet$}
        \put(7.5,-0.5){$\bullet$}
        \put(4.5,0){$\cs$}
        \put(2.5,0){$\cs$}
        \put(4.5,2){$\cs$}
        \put(23.5,-2){${}^{\Da_2}$}
        \put(7.5,-2){${}^{\Da_1}$}
    \end{picture}
    \caption{ The branch of an order $\Ha$.}\label{f1}
    \label{f2}
\end{figure}
In the figure, clubs and spades denote the 
maximal orders containing two Eichler orders
of level $2$ that contains the order $\Ha$. 
The relative position of these Eichler
orders is relevant, for example, when solving the selectivity problem, see 
\cite{Eichler2} for an account on that subject.

We proved in \cite{A26} that an order $\Ra$ in a matrix algebra 
can be written as 
a ghost intersection of maximal orders, i.e., as a finite intersection
$$\Ra=\alge\cap\bigcap_{j\in J}\Da_j,$$
where each $\Da_j$ in a maximal order is 
$\alge_L=\alge\otimes_KL$, for some finite
extension $L/K$, if and only if the Gorenstein closure
of $\Ra$ is a Bass order. This is interesting because 
embeddings of orders into
these ghost intersections can be studied via graph theory through the
theory of ghost branches that we developed in \cite{GB1}. The preceding 
characterization of Bass orders allows us to extend this 
graph theoretical technique
to that setting, and thus, for instance, potentially complete current computations of spinor images 
that are key to solving the selectivity problem. Having a similar result over division 
algebras can be useful in that direction, since we recently managed to extend
the theory of ghost branches to this setting via twisted trees. Essentially, a twisted
tree is a Bruhat-Tits tree with a modified Galois action having no invariant visual 
limits. See \cite{GB2} for details. Thus, the main result of this paper is the 
following:

\begin{theorem}\label{t1}
    A full order in a quaternion division algebra over a local field $K$ is a ghost 
    intersection of maximal orders if and only if its Gorenstein closure is a Bass 
    order.
\end{theorem}

An immediate consequence is the following result:

\begin{corollary}
    If $\{\Ra_i\}_{i\in I}$ is a finite family of orders having a Bass 
    Gorenstein closure, then the intersection $\bigcup_{i\in I}\Ra_i$ has the
    same property.\qed
\end{corollary}

In general, we identify $\alge$ with a subset of $\alge_L$, for instance, when writing intersections. However, we still write $a_L\in\alge_L$ for an element $a\in \alge$.
The expression $a_L$ should be interpreted as \emph{$a$ as an element
of $\alge_L$}. This is relevant, for instance, in Prop. \ref{p25}.

\section{Quaternion algebras, branches and extensions}

In this section $\algeq$ is an arbitrary quaternion $K$-algebra, i.e., either the
matrix algebra $\alge_0\cong\mathbb{M}_2(K)$ or the unique quaternion division algebra
$\alge$, up to isomorphism. 
In either case we identify $K$ with 
$K\mathbf{1}\subset\algeq$, where 
$\mathbf{1}=\mathbf{1}_\algeq\in\algeq$ is the
multiplicative identity,
and write $\algeq_*=\alge\smallsetminus K$.
However, we still write scalar quaternions $a\mathbf{1}$
for emphasis. If the characteristic of $K$ is not $2$, 
the division algebra is defined by 
$$\alge=\left(\frac{\pi_K,\delta}{K}\right)=
K\left[\mathbf{i},\mathbf{j}\big|\mathbf{i}^2=\pi_K\mathbf{1},
\ \mathbf{j}^2=\delta\mathbf{1},\ \mathbf{i}\mathbf{j}=
-\mathbf{j}\mathbf{i}\right],$$
where $\delta$ is a unit of maximal quadratic defect.
For a field of characteristic $2$ it is defined instead by the formula 
$\alge=K\left[\mathbf{i},\mathbf{j}\big|\mathbf{j}^2+\mathbf{j}=z\mathbf{1},\ 
\mathbf{i}^2=\pi_K\mathbf{1},\ \mathbf{j}\mathbf{i}=(\mathbf{i}+\mathbf{1})\mathbf{j}\right]$, where $z\in K$ is 
not in the image of the Artin–Schreier map. In either case, the quaternion division
algebra $\alge$ has a unique maximal order $\mathcal{O}_{\alge}=\mathcal{O}_K[\mathbf{i},\mathbf{w}]$, where $\mathbf{w}$ is a generator of the ring of integers in $K[\mathbf{j}]$.
On the other hand, the matrix algebra $\alge_0$ has infinitely many maximal orders,
and they are naturally in correspondence with the vertices of the Bruhat-Tits tree.
The vertex corresponding to a maximal order $\Da$ is denoted $v_{\Da}$. The
vertices $v_{\Da}$ and $v_{\Da'}$ are neighbors if, in some basis, 
the corresponding orders have the form
$$\Da=\bmattrix{\mathcal{O}_K}{\mathcal{O}_K}{\mathcal{O}_K}{\mathcal{O}_K},
\qquad\Da'=\bmattrix{\mathcal{O}_K}{\mathcal{O}_K}{\pi_K\mathcal{O}_K}
{\mathcal{O}_K}.$$
The Bruhat-Tits tree is a homogeneous tree where every vertex has valency $q+1$,
if $q$ is the cardinality of the residue field of $K$.
For every element $a\in\alge_0$,
the branch $\mathfrak{s}(a)=\mathfrak{s}_K(a)$ is the 
largest subgraph whose vertices
correspond to maximal orders containing $a$. The branch of an order $\Ha$ can be
defined as $\mathfrak{s}(\Ha)=\bigcap_{a\in\Ha}\mathfrak{s}(a)$. 
The list of possible shapes for the branch of an element 
$a\in\alge_{0*}=\alge_0\smallsetminus K$ is rather
limited, so we give the full list below. All these branches
are connected and so are their intersections as they are intersections in a tree, where paths are unique. The 
branch of the order $\Ha^{[r]}:=\mathcal{O}_K+\pi_K^r\Ha$ is 
constructed adding to the branch $\mathfrak{s}=\mathfrak{s}(\Ha)$ every vertex at distance $r$ or less
\cite[Prop. 2.4]{Eichler2}.  This graph is called the $r$-tubular neighborhood 
$\mathfrak{s}^{[r]}$ of $\mathfrak{s}$. This fact can be
used to reduce computations to the case where $\mathcal{O}_K[a]$ is the maximal order of $K[a]$
whenever the latter is a semisimple algebra.
\begin{itemize}
    \item If $\mathcal{O}_K[a]$ is the ring of integer of an algebra  
    $K[a]\cong K\times K$, then $\mathfrak{s}(a)$ is a maximal path.
    \item If $\mathcal{O}_K[a]$ is the ring of integer of a unramified field, then
    $\mathfrak{s}(a)$ contains a unique vertex.
    \item If $\mathcal{O}_K[a]$ is the ring of integer of a ramified field, then
    $\mathfrak{s}(a)$ is an edge.
    \item If $K(a)$ contains a nilpotent element, then it
    is not semisimple. In this case $\mathfrak{s}(a)$ is a horoball. 
\end{itemize} 
    Above, a horoball is a connected subgraph $\mathfrak{h}$
    of $\mathfrak{t}(K)$ having leaves (valency one vertices)
    and being isomorphic as a graph to $\mathfrak{h}^{[1]}$.
    This property is enough to completely characterize horoballs, and this fact lies at the heart of the proof
    of \cite[Prop. 4.4]{Eichler2}. We call horoballs ``infinite foliages'' in that work.
    \begin{figure}
    \centering
    \setlength{\unitlength}{2mm}
    \begin{picture}(100,15)(0,-3)
        \put(55.5,-0.1){$\dots$}
        \put(54.7,-0.3){$\star$}
        \put(66,8){\textbf{D}}
        \put(58,0){\line(1,0){15}}
        \put(61,0){\line(0,1){10}}
        \put(61,5){\line(1,0){5}}
        \put(61,8){\line(1,0){2}}
        \put(64,5){\line(0,1){2}}
        \put(68,0){\line(0,1){5}}
        \put(68,3){\line(1,0){2}}
        \put(71,0){\line(0,1){2}}
        \put(14,8){\textbf{A}}
        \put(4.5,-0.1){$\dots$}
        \put(3.7,-0.3){$\star$}
        \put(21.5,-0.1){$\dots$}
        \put(23.5,-0.3){$\star$}
        \put(7,0){\line(1,0){14}}
        \put(10,0){\line(0,1){5}}
        \put(10,3){\line(1,0){2}}
        \put(14,0){\line(0,1){5}}
        \put(14,3){\line(1,0){2}}
        \put(18,0){\line(0,1){5}}
        \put(18,3){\line(1,0){2}}
        \put(9.5,-0.5){$\bullet$}
        \put(13.5,-0.5){$\bullet$}
        \put(17.5,-0.5){$\bullet$}
        \put(44,8){\textbf{C}}
        \put(39,0){\line(1,0){14}}
        \put(41,0){\line(0,1){2}}
        \put(44,0){\line(0,1){5}}
        \put(44,3){\line(1,0){2}}
        \put(48,0){\line(0,1){5}}
        \put(48,3){\line(1,0){2}}
        \put(51,0){\line(0,1){2}}
        \put(43.5,-0.5){$\bullet$}
        \put(47.5,-0.5){$\bullet$}
        \put(31,8){\textbf{B}}
        \put(27,0){\line(1,0){10}}
        \put(29,0){\line(0,1){2}}
        \put(32,0){\line(0,1){5}}
        \put(32,3){\line(1,0){2}}
        \put(35,0){\line(0,1){2}}
        \put(31.5,-0.5){$\bullet$}
    \end{picture}
    \caption{Possible branches $\mathfrak{s}(a)$. \textbf{A}-\textbf{C},
    show the tubular neighborhoods of a maximal path, a vertex and an
    edge with their stem vertices denoted by bullets. 
    \textbf{D} shows a horoball. We assume the residue field 
    has 2 elements. Stars denote visual limits.}
    \label{f2}
\end{figure}
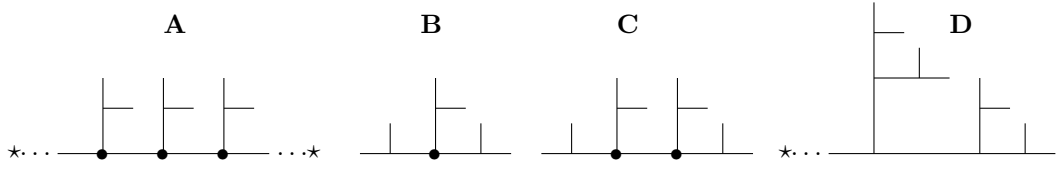

When $\mathcal{O}_K[a]$ is not maximal in the semisimple algebra $K[a]$, 
the branch $\mathfrak{s}(a)$ is a tubular neighborhood of one of the preceding
shapes. More precisely, we write $\mathcal{O}_K[a]=
\mathcal{O}_{K[a]}^{[n]}$, where $\mathcal{O}_{K[a]}$ is the maximal order
in $K[a]$. In this case the maximal orders containing $\mathcal{O}_{K[a]}$
are called the stem orders of $\mathfrak{s}(a)$, and their corresponding
vertices are called stem vertices. The maximal orders
containing $a$, i.e., those corresponding to vertices in $\mathfrak{s}(a)$,
are precisely those corresponding to vertices at distance $n$
or less from a stem vertex.

Once intersections are added to the
list, only two more graph shapes need to be included. A tubular neighborhood of an 
arbitrary finite path, for full orders, or a tubular neighborhood of a ray for three dimensional orders,
see \cite{Eichler2}.

\begin{Def}\label{d1}
    Let $\algeq$ be either $\alge_0$ or $\alge$, let 
    $\mathfrak{L}\subset\algeq$ be a two-dimensional semisimple commutative 
    algebra whose unique maximal order is $\mathcal{O}_{\mathfrak{L}}$, let 
    $\mathbf{H}\subset\algeq$ be a full order,
    and let $\algeq_*=\algeq\smallsetminus K$, as before. 
    Let $n=n(\mathfrak{L},\Ha)$ be the 
    non-negative integer satisfying 
    $\mathfrak{L}\cap\Ha=\mathcal{O}_{\mathfrak{L}}^{[n]}$. For any semisimple 
    element $a\in\algeq_*$ and for any full 
    order $\mathbf{H}$,
    we write $n(a,\mathbf{H})$ for $n(K[a],\mathbf{H})$. We also set $m(a)=m$ for 
    the integer satisfying $\mathcal{O}_K[a]=\mathcal{O}_{K[a]}^{[m]}$. Note that
    $a\in\mathbf{H}$ is equivalent to $m(a)\geq n(a,\mathbf{H})$.
\end{Def}

\begin{proposition}
    Every element $a\in\alge_*$ is semisimple.
    Two full orders $\Ha$ and $\Ha'$ in $\algeq=\alge$ are equal if and only if
    $n(a,\Ha)=n(a,\Ha')$ for any element $a\in\algeq_*$.
\end{proposition}

\begin{proof}
    The first statement is well known, since every element
    $a\in\alge_*$ generates a field. 
    By the last statement in Def. \ref{d1}, the condition $n(a,\Ha)=n(a,\Ha')$
    implies that $a\in\mathbf{H}$ if and only if $a\in\mathbf{H}'$. The result
    follows since $K\cap\Ha=\mathcal{O}_K$ for every order $\Ha$.
\end{proof}

\begin{proposition}
    Let $\mathfrak{L}\subset\alge_0$ be a two dimensional semisimple 
    commutative algebra and let $m\geq0$ be arbitrary. 
    Then $n(\mathfrak{L},\mathbf{D})=0$, for a maximal order $\mathbf{D}$, if and only
    if $\mathbf{D}$ is a stem order of $\mathcal{O}_{\mathfrak{L}}^{[m]}$.
    Furthermore, if two neighboring maximal orders, $\mathbf{D}$ and $\mathbf{D}'$
    satisfy $n(\mathfrak{L},\mathbf{D})=n(\mathfrak{L},\mathbf{D}')$ then 
    $n(\mathfrak{L},\mathbf{D})=0$. If $n(\mathfrak{L},\mathbf{D})=n>0$, then
    $\mathbf{D}$ has a unique neighbor $\mathbf{D}'$ such that 
    $n(\mathfrak{L},\mathbf{D}')=n-1$. Any other neighbor $\mathbf{D}''$ satisfies 
    $n(\mathfrak{L},\mathbf{D}'')=n+1$.
\end{proposition}

\begin{proof}
    The first statement is immediate from the definition of stem order. 
    Note that the stem vertices are independent of $m$, and are the vertices
    of a path, which is a subtree. For any maximal order that is not a stem order
    of $\mathcal{O}_{\mathfrak{L}}^{[m]}$, there is a unique path connecting it
    to the closest stem vertex. This path goes through a unique neighbor,
    which might be a stem vertex itself. This neighbor is one unit closer to the
    set of stem vertices, while all the others are farther away. In particular,
    no neighbor is at the same distance.
\end{proof}

\begin{corollary}
    Let $\mathfrak{L}\subset\alge_0$ be a two dimensional semisimple 
    commutative algebra. Consider a sequence $\mathbf{D}_0,\mathbf{D}_1,\dots,\mathbf{D}_n$ of maximal orders
    satisfying the following conditions:
    \begin{itemize}
        \item $n(\mathfrak{L},\mathbf{D}_0)=0$,
        \item $\mathbf{D}_i$ and $\mathbf{D}_{i+1}$ are neighbors for $0\leq i\leq n-1$,
        and 
        \item $\mathbf{D}_i\neq\mathbf{D}_{i+2}$ for $0\leq i\leq n-2$.
    \end{itemize}
    Then, there is an integer $r$ with $0\leq r\leq n$ such that $n(\mathfrak{L},\mathbf{D}_i)=0$, 
    for $0\leq i\leq r$, and 
    $n(\mathfrak{L},\mathbf{D}_i)=i-r$, for $r\leq i\leq n$.
\end{corollary}

\begin{proof}
    Since $\mathbf{D}_i$ and $\mathbf{D}_{i+1}$ are neighbors, it is clear that
    $n(\mathfrak{L},\mathbf{D}_i)$ and $n(\mathfrak{L},\mathbf{D}_{i+1})$ differ
    by $1$ at most, and they are equal if and only if they are both $0$. Let $r$
    be the largest integer such that $n(\mathfrak{L},\mathbf{D}_i)=0$ for
    $0\leq i\leq r$. If $r=n$, there is nothing to prove. If $r<n$, then
    $n(\mathfrak{L},\mathbf{D}_{r+1})=1$. If $r<n-1$, we must have
    $n(\mathfrak{L},\mathbf{D}_{r+2})=2$ since there is exactly one neighbor $\Da$
    of $\Da_{r+1}$ satisfying $n(\mathfrak{L},\Da)=0$, and it is 
    $\Da=\Da_r$. Iterations are analogous.
\end{proof}

\begin{proposition}\label{p25}
    Let $E/K$ be a finite extension and let $e=e(E/K)$ be its ramification degree.
    Let $a\in\algeq_*$ be a semisimple element whose image 
    $a_E\in \algeq_E$ is also semisimple. 
    Then, we have $m(a_E)=em(a)+b$
    for some constant $b=b(\mathfrak{L})$ depending 
    only on the isomorphism class of $\mathfrak{L}=K[a]$. Furthermore, if 
    $n=n(a_E,\mathbf{S})$ and $n_1=n(a,\algeq\cap\mathbf{S})$,
    for a full order $\mathbf{S}\subset\algeq_E$, then $e(n_1-1)+b<n\leq en_1+b$.
\end{proposition}

\begin{proof}
    Since every order in $\mathfrak{L}_E$ has the form 
    $\mathcal{O}_{\mathfrak{L}_E}^{[s]}$, for some $s\in\mathbb{Z}_{\geq0}$, we can 
    define $b$ by the relation $(\mathcal{O}_{\mathfrak{L}})_E=
    \mathcal{O}_{\mathfrak{L}_E}^{[b]}$. The first statement follows since 
    $\pi_K=u\pi_E^e$ for some unit $u\in\mathcal{O}_E^*$. For the second statement,
    note that $K[a]$ has integral generators $a'$ and $a''$ satisfying $m(a')=n-1$
    and $m(a')=n$. This means that $a'\in\algeq\cap\mathbf{S}$, while $a''\notin\algeq\cap\mathbf{S}$. This implies $a_E'\in\mathbf{S}$ and
    $a_E''\notin\mathbf{S}$. The result follows.

\end{proof}

\section{Proof of the main result}

In this section we collect the main tools that we need
to prove Th.~\ref{t1}. 

\begin{lemma}
    A full order $\Ba\subseteq\alge$ is Bass if and only if $\Ba\not\subset
    \mathcal{O}_{\alge}^{[1]}$.
\end{lemma}

\begin{proof}
    If $\Ba\subseteq\mathcal{O}_{\alge}^{[1]}$, then $\Ba$ is not Bass by
    definition, since $\mathcal{O}_{\alge}^{[1]}$ is not Gorenstein. On the
    other hand, if $\Ba$ is not Bass, then $\Ba\subset\Ra^{[1]}$ for
    some order $\Ra$, but $\Ra^{[1]}\subseteq\mathcal{O}_{\alge}^{[1]}$,
    since $\Ra\subseteq\mathcal{O}_{\alge}$.
\end{proof}

\begin{lemma}\label{l32}
    Let $L/K$ be the unique unramified quadratic extension.
    If $\Ba\subseteq\alge$ is a Bass order, the $\mathcal{O}_L$-order $\Ba_L\subset
    \alge_L$ generated by $\Ba$ is a Bass order.
\end{lemma}

\begin{proof}
    Since $\Ba$ is Bass, it follows from the preceding lemma that there is an element
    $b\in\Ba$ that does not belong to $\mathcal{O}_{\alge}^{[1]}$. In particular, the
    order $\Ha=\mathcal{O}_K[b]\subset K[b]$ is maximal in the latter subfield. Choosing
    a different generator of the commutative order if needed, we can assume that $b$
    is either a uniformizing parameter in a ramified quadratic extension, or a unit
    generating an unramified quadratic extension whose image on the residue field satisfies $\mathbb{K}[\bar{b}]\neq\mathbb{K}$.

    In the first case, $b_L$ is still a uniformizing parameter in a ramified quadratic
    extension $L[b_L]$ of $L$. It follows from  \cite[Prop. 4.2]{Eichler2} that
    its branch $\mathfrak{s}_L(b_L)$ is an edge. In particular, the branch $\mathfrak{s}_L(\Ba_L)$
    is contained in an edge. It follows from \cite[Th. 1.2]{A26}
    that $\mathfrak{s}_L(\Ba_L)$ is Bass.

    In the second case we have $L[b_L]\cong L\times L$, since the eigenvalues of $b$ are
    in $L$. Under this isomorphism $b_L$ corresponds to a pair of the form 
    $(\lambda,\bar{\lambda})$, where $\bar{\lambda}$ is the image of $\lambda$
    under the non-trivial $K$-automorphism of $L$. Since $\mathbb{K}[\bar{b}]\neq
    \mathbb{K}$, $\nu_L(\lambda-\bar{\lambda})=0$, which implies by 
    \cite[Prop. 4.2]{Eichler2} that $\mathfrak{s}_L(b)$ is a maximal path. Now the
    result follows as before.
\end{proof}

\begin{lemma}
    Let $L/K$ be the unique unramified quadratic extension. Then for any full order
    $\Ra\subset\alge$ we have $(\Ra^{[r]})_L=(\Ra_L)^{[r]}$. In particular, if the
    Gorenstein closure of $\Ra$ is a Bass order, the same holds for $\Ra_L$.
\end{lemma}

\begin{proof}
    Since $L/K$ is unramified, $\pi_K$ is also a uniformizing parameter
    for $L$. It follows that 
    $$(\Ra^{[r]})_L=(\mathcal{O}_K+\pi_K^r\Ra)_L=
    \mathcal{O}_L(\mathcal{O}_K+\pi_K^r\Ra)=
    \mathcal{O}_L+\pi_K^r(\mathcal{O}_L\Ra)=(\Ra_L)^{[r]}.$$
    For the last statement, if the Gorenstein closure $\Ra_0$
    of $\Ra$ is Bass, then $\Ra=\Ra_0^{[m]}$ for some integer
    $m\geq0$. Furthermore, $(\Ra_0)_L$ is Bass by the preceding result, and we have $\Ra_L=(\Ra_0^{[m]})_L=
    (\Ra_0)^{[m]}_L$, whence the result follows. 
\end{proof}

\begin{proposition}\label{p34}
    If the Gorenstein closure of $\Ra$ is a Bass order, then $\Ra$ is a ghost 
    intersection of maximal orders.
\end{proposition}

\begin{proof}
    Since the Gorenstein closure of $\Ra$ is a Bass order, the same holds for $\Ra_L$
    by the preceding lemma. Since $\alge_L$ is a matrix algebra, then $\Ra_L$
    is a ghost intersection of maximal orders by \cite[Th. 1.1]{A26}. 
    This means that $\Ra_L=\alge_L\cap\bigcap_{j\in J}\Da_j$ for some collection
    $\{\Da_j|j\in J\}$ of maximal orders in $\alge_E$, for some extension $E/L$.
    Then $$\Ra=\alge\cap\Ra_L=\alge\cap\bigcap_{j\in J}\Da_j,$$
    as claimed.
\end{proof}

\begin{proposition}\label{p35}
    Write $\Ra=\Ra_1^{[r]}$, with $r>0$, for a non-Gorenstein order $\Ra$. With this notation, if $\Ra$
    is a ghost intersection of maximal orders, then
    so is $\Ra_1$.
\end{proposition}

\begin{proof}
    By induction, it suffices to assume $r=1$.  Note that 
    $\Ra\subseteq\mathcal{O}_{\alge}^{[1]}$, whence,
    $n(a,\Ra)\geq 1$ for any $a\in\alge_*=
    \alge\smallsetminus K$. Write 
    $\Ra=\alge\cap\mathbf{S}$, where $\mathbf{S}=\bigcap_{j\in J}\Da_j$
    is a full order in $\alge_E$ that can be written as the intersection of some
    collection $\{\Da_j\}_{j\in J}$ of maximal orders in $\alge_E$ 
    for some finite 
    extension $E/K$.  Fix a maximal order $\Da$ in $\alge_E$ containing 
    $\mathcal{O}_{\alge}$. Then $n(a_E,\mathbf{D})\leq b(a)$,
    if $b(a)=b(K\big(a)\big)$ as in Prop. \ref{p25}, for any 
    $a\in\alge_*$. Let $e=e(E/K)$. For every order $\Da_j$ 
    in the collection we choose a maximal order $\Da'_j$ as follows: 
    \begin{itemize} \item If the distance
    in the Bruhat-Tits tree of $\alge_E$ between the vertices corresponding to
    $\Da_j$ and $\Da$ is $e$ or less, we set $\Da'_j=\Da$,
    \item otherwise we choose $\Da'_j$ as the maximal order corresponding to a vertex
    in the path between $v_{\Da}$ and $v_{\Da_j}$ at distance $e$ from
    $v_{\Da_j}$. \end{itemize}
    Let $\mathbf{S}'=\bigcap_{j\in J}\Da'_j$.
    It suffices to show that $n(a,\alge\cap\mathbf{S})=n(a,\alge\cap\mathbf{S}')+1$
    for every $a\in\alge_*$. To prove this we note that 
    $n(a_E,\mathbf{D})\leq b(a)$. For any $j\in J$, we have two cases:
    \begin{itemize} \item If $n(a_E,\mathbf{D}'_j)\leq b(a)$, then
    $n(a_E,\Da_j)\leq b(a)+e$ since $d(v_{\Da},v_{\Da_j})\leq e$.
    In this case $n(a,\alge\cap\Da_j)\leq 1$ and $n(a,\alge\cap\Da'_j)=0$.
    \item If $n(a_E,\mathbf{D}'_j)> b(a)\geq0$, then $\mathbf{D}'_j\neq
    \mathbf{D}$, whence $d(v_{\Da'_j},v_{\Da_j})=e$. Then $n(a_E,\Da_j)=
    n(a_E,\Da'_j)+e$. This implies $n(a,\alge\cap\Da_j)=
    n(a,\alge\cap\Da'_j)+1$.
    \end{itemize}
    We conclude that $n(a,\alge\cap\Da_j)=n(a,\alge\cap\Da'_j)+1$ unless
    $n(a,\alge\cap\Da_j)=n(a,\alge\cap\Da'_j)=0$. The condition $n(a,\Ra)\geq 1$
    show that these values of $j$ are not relevant to compute $n(a,\Ra)=
    \max_{j\in J}n(a,\alge\cap\Da_j)$, and they can be ignored. The result follows.      
\end{proof}

\begin{proposition}\label{p36}
    Let $\Ra$ be a Gorenstein order that is a ghost intersection of maximal orders. Then 
    $\Ra$ is a Bass order.
\end{proposition}

\begin{proof}
    If $\Ra$ is not a Bass order, then $\Ra\subseteq\mathcal{O}_{\alge}^{[1]}$. 
    Write $\Ra=\alge\cap\mathbf{S}$, where $\mathbf{S}=\bigcap_{j\in J}\Da_j$, 
    and define $\mathbf{S}'$ as in the preceding proof. Then the same 
    computation as before shows $n(a,\Ra)=n(a,\alge\cap\mathbf{S}')+1$ for 
    all $a\in\alge_*$. It follows that 
    $\Ra_1=\alge\cap\mathbf{S}'$ satisfies $\Ra=\Ra_1^{[1]}$. This 
    contradicts the fact that $\Ra$ is Gorenstein, concluding the proof.
\end{proof}

\begin{proof}[Proof of Theorem~\ref{t1}]
    If the Gorenstein closure of $\Ra$ is Bass, it follows from Prop.
    \ref{p34} that $\Ra$ is a ghost intersection of maximal orders. 
    On the other hand, if $\Ra$ is a ghost intersection
    of maximal orders, so is its Gorenstein closure
    $\Ra'$ by Prop. \ref{p35}. Now Prop. \ref{p36} shows that $\Ra'$
    is Bass.
\end{proof}

\section{Examples}

In all of this section, by a module we mean a left module, and by ideal
in a ring we mean a left ideal unless otherwise stated.
Next result is used throughout the examples:

\begin{lemma}
 Let R be a ring, and let $\Lambda_0 \subset \Lambda_1$ be two R-modules 
 satisfying $\Lambda_1/\Lambda_0\cong R$ as $R$-modules. Then every 
 $R$-submodule of $\Lambda_1$ containing $\Lambda_0$ equals $I\Lambda_1 
 + \Lambda_0$ for some ideal $I$ of $R$. In particular, if
$R$ is a principal ideal domain, then every $R$-submodule of $\Lambda_1$ 
containing $\Lambda_0$ equals $a\Lambda_1 + \Lambda_0$, for some
element $a\in R$.    
\end{lemma}

\begin{proof}
    It is immediate from the correspondence theorem that all the considered  
    submodules have the form $\Lambda_2/\Lambda_0$, where 
    $\Lambda_0\subseteq\Lambda_2\subseteq\Lambda_1$. Since the 
    $R$-submodules of $R$ are precisely the ideals, the result follows from 
    the identity $I(\Lambda_1/\Lambda_0)=I\Lambda_1+\Lambda_0$.
\end{proof}

\begin{example}\label{e1}
    Let $L/K$ be an unramified quadratic extension. Let $\alge$ be a 
    quaternion division algebra, as before, so that $\mathcal{O}_\alge=
    \mathcal{O}_K[\mathbf{i},\mathbf{w}]$, where 
    $\mathbf{i}^2=\pi_K\mathbf{1}$, and
    $\mathbf{w}$ is a generator of the ring of integers of an unramified
    field $\mathfrak{L}\subseteq\alge$, on which conjugation by $i$
    induces the nontrivial element in the Galois group 
    $\mathrm{Gal}(\mathfrak{L}/K)$.
    Then $L\cong\mathfrak{L}$ as $K$-algebras, but we use different
    letter type to emphasize the particular roles each plays in the 
    computation. In fact, if $\omega\in L$ is the image of $\mathbf{w}$
    under a suitable isomorphism, then the algebra $\mathfrak{L}_L
    =L\otimes_K\mathfrak{L}$ has the idempotent
    $$\frac{1}{\bar{\omega}-\omega}\otimes\mathbf{w}-
    \frac{\omega}{\bar{\omega}-\omega}\otimes\mathbf{1},$$
    where $\bar{\omega}$ is the conjugate of $\omega$. It follows that
    $\mathfrak{L}_L$ is isomorphic to $L\times L$, and we can assume that
    the isomorphism takes $\mathbf{w}_L=1\otimes\mathbf{w}$ to $(\omega,\bar{\omega})$.
     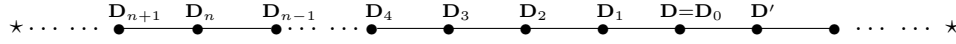
\begin{figure}
    \centering
    \setlength{\unitlength}{2mm}
    \begin{picture}(70,5)(0,-3)
        \put(59.6,-0.1){$\dots$}
        \put(62.6,-0.1){$\dots$}
        \put(65.6,-0.3){$\star$}
        \put(28,0){\line(1,0){30}}
        \put(4.9,-0.1){$\dots$}
        \put(7.6,-0.1){$\dots$}
        \put(3.7,-0.3){$\star$}
        \put(22,-0.1){$\dots$}
        \put(24.8,-0.1){$\dots$}
        \put(11,0){\line(1,0){10}}
        \put(10.5,-0.5){$\bullet$}
        \put(15.7,-0.5){$\bullet$}
        \put(20.9,-0.5){$\bullet$}
        \put(27.2,-0.5){$\bullet$}
        \put(32.3,-0.5){$\bullet$}
        \put(37.4,-0.5){$\bullet$}
        \put(42.5,-0.5){$\bullet$}
        \put(47.6,-0.5){$\bullet$}
        \put(52.7,-0.5){$\bullet$}
        \put(57.8,-0.5){$\bullet$}
        \put(52.7,1){${}_{\Da'}$}
        \put(46.7,1){${}_{\Da=\Da_0}$}
        \put(42.5,1){${}_{\Da_1}$}
        \put(37.4,1){${}_{\Da_2}$}
        \put(32.3,1){${}_{\Da_3}$}
        \put(27.2,1){${}_{\Da_4}$}
        \put(20.4,1){${}_{\Da_{n-1}}$}
        \put(15.3,1){${}_{\Da_n}$}
        \put(10.2,1){${}_{\Da_{n+1}}$}
    \end{picture}
    \caption{The maximal path whose vertices correspond to maximal orders
    containing the matrix $\mathbf{w}_L$ in Ex. \ref{e1}.}
    \label{f2}
\end{figure}
    This implies that $\omega$ and $\bar{\omega}$ are the 
    eigenvalues of the matrix $\mathbf{w}_L$, and since they have 
    different images in the residue
    field, $\mathbf{w}_L$ is 
    contained in the maximal orders
    in an infinite path by 
    \cite[Prop. 4.2]{Eichler2}. By a similar argument,
    the matrix $\mathbf{i}_L$, 
    which generates the ring of integers
    of a ramified extension, is contained 
    in two maximal orders, and 
    conjugation by $\mathbf{i}_L$ transposes them. As conjugation by
    $\mathbf{i}_L$ also transposes the eigenvalues of 
    $\mathbf{w}_L$, and therefore the visual limits of the 
    corresponding path, it follows that the maximal orders containing
    $\mathbf{i}_L$ are the two maximal orders, $\Da$ and $\Da'$ 
    that are closest to the unique $\mathbf{i}_L$-invariant point 
    in the path, which is the barycenter of an edge, see Fig.\ref{f2}.
    
    Now, for each $n\geq 0$, let $\Da_n$ be a maximal order in $\alge_L$
    whose corresponding vertex lies on the path at a distance of $n$ from 
    the vertex $\Da=\Da_0$, and at distance $n+1$ from $\Da'$,
    as in Fig. \ref{f2}. The
    order $\Da_n$ contains $\pi_K^n\mathbf{i}_L$, but not
    $\pi_K^{n-1}\mathbf{i}_L$. Since the orders of the form
    $\Ba_n=\mathcal{O}_K[\pi_K^n\mathbf{i},\mathbf{w}]$ are the only ones 
    containing $\mathcal{O}_{\mathfrak{L}}=\mathcal{O}_K[\mathbf{w}]$,
    by the preceding lemma, we have $\Ba_n=\alge\cap\Da_n$.
\end{example}

\begin{example}\label{e2}
    Now assume the characteristic of $K$ is not $2$, and let
    $F=K[\sqrt{\pi_K}]$, which is a ramified quadratic extension.  As 
    before, $\alge$ is a division algebra, while $\alge_F$ is a matrix
    algebra where $$\mathtt{p}=\frac{1}{2\sqrt{\pi_K}}\otimes\mathbf{i}-
    \frac12\otimes\mathbf{1},$$ is an idempotent. In other words, 
    the algebra $\mathfrak{F}=K[\mathbf{i}]$ is a field isomorphic to $F$,
    while $\mathfrak{F}_F\cong F\times F$. Furthermore,
    $\mathcal{O}_F[\mathbf{w}_F]$ is the ring of integers
    of an unramified extension, whence it is contained in a unique
    maximal order $\Da$. The same holds, therefore, for 
    $(\mathcal{O}_\alge)_F$. Now the maximal orders containing 
    $\mathtt{p}$ lie in a maximal path $\mathfrak{p}$
    see Fig. \ref{f3nn}.
    As $\mathbf{i}'=\frac{1}{\sqrt{\pi_K}}\otimes\mathbf{i}=2\mathtt{p}+
    \mathbf{1}_L$ and $\mathbf{w}_L$ generate a cyclic
    algebra (and therefore a matrix algebra) over the residue field,
    they must generate a maximal order $\Da$. We conclude that $\Da$ 
    lies at a distance of $\nu_F(2)=2\nu_K(2)$ from $\mathfrak{p}$. Let 
    $\Da'$ be the order of the line that is closest to $\Da$. Note that 
    $\Da=\Da'$ when $K$ is not dyadic. Let $\Da_n$ be an order at distance 
    $n$ from $\Da$, either on the path from $\Da$ to $\Da'$ (when 
    $n\leq\nu_F(2)$), or in the path $\mathfrak{p}$ when $n\geq\nu_F(2)$).
    The order $\Da_n$ contains $\sqrt{\pi_K}^n\mathbf{w}_L$
    but no $\sqrt{\pi_K}^{n-1}\mathbf{w}_L$. 
    However, all of them contain  $\mathbf{i}'$, which is a unit. We can 
    rephrase this as follows:
    \begin{itemize}
    \item When $n=2t$ then $\Da_n$ contains
    $\pi_K^t\mathbf{w}_L=
    \mathbf{i}_L^n\mathbf{w}_L$,
    but not $\pi_K^{t-1}\mathbf{i}_L\mathbf{w}_L=
    \mathbf{i}_L^{n-1}\mathbf{w}_L$. 
    \item When $n=2t+1$ then $\Da_n$ contains 
    $\pi_K^t\mathbf{i}_L\mathbf{w}_L=
    \mathbf{i}_L^n\mathbf{w}_L$, 
    but not $\pi_K^t\mathbf{w}_L=\mathbf{i}_L^{n-1}\mathbf{w}_L$. 
     \end{itemize}
    We conclude as before that $D_n\cap\alge=\mathcal{O}_K[\mathbf{i}^n\mathbf{w},\mathbf{i}]$,
    since all maximal orders containing $\mathbf{i}$ have this form.
     \begin{figure}
    \centering
    \setlength{\unitlength}{2mm}
    \begin{picture}(50,15)(20,-12)
        \put(59.6,-0.1){$\dots$}
        \put(62.6,-0.1){$\dots$}
        \put(65.6,-0.3){$\star$}
        \put(28,0){\line(1,0){30}}
        \put(20.7,-0.3){$\star$}
        \put(22,-0.1){$\dots$}
        \put(24.8,-0.1){$\dots$}
        \put(27.2,-0.5){$\bullet$}
        \put(32.3,-0.5){$\bullet$}
        \put(37.4,-0.5){$\bullet$}
        \put(42.5,-0.5){$\bullet$}
        \put(47.6,-0.5){$\bullet$}
        \put(52.7,-0.5){$\bullet$}
        \put(57.8,-0.5){$\bullet$}
        \put(46.7,1){${}_{\Da'=\Da_\nu}$}
        \put(42.5,1){${}_{\Da_{\nu+1}}$}
        \put(37.4,1){${}_{\Da_{\nu+2}}$}
        \put(32.3,1){${}_{\Da_{\nu+3}}$}
        \put(27.2,1){${}_{\Da_{\nu+4}}$}
        \put(48,-2){\line(0,1){2}}
        \put(48,-10){\line(0,1){4.8}}
        \put(47.8,-4.5){$\vdots$}
        \put(47.6,-2.1){$\bullet$}
        \put(47.6,-5.7){$\bullet$}
        \put(47.6,-7.9){$\bullet$}
        \put(47.6,-10.1){$\bullet$}
        \put(49,-2){${}_{\Da_{\nu-1}}$}
        \put(49,-5.7){${}_{\Da_{2}}$}
        \put(49,-7.9){${}_{\Da_{1}}$}
        \put(49,-10.1){${}_{\Da=\Da_{0}}$}
    \end{picture}
    \caption{The maximal orders in Ex. \ref{e2}.}
    \label{f3nn}
\end{figure}
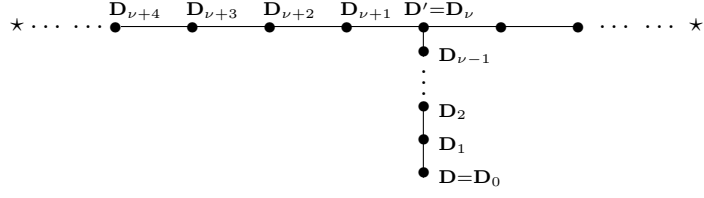
\end{example}



  Luis Arenas-Carmona\\
Universidad de Chile,\\ Facultad de Ciencias, \\Casilla 653, Santiago,\\
   Chile \\
   \email{learenas@u.uchile.cl}
\end{document}